\documentclass[11pt]{article}

\usepackage{amsmath}
\usepackage{amsthm}
\usepackage{amssymb}
\usepackage{amscd}
\usepackage{subfig}
\usepackage{amsfonts}
\usepackage{amsbsy}
\usepackage{epsfig,afterpage}
\usepackage[dvips]{psfrag}
\usepackage{epsf}
\usepackage{psfrag}
\usepackage{color}     
\usepackage{graphicx}
\usepackage{epstopdf}
\usepackage[utf8]{inputenc}    
\usepackage{overpic}

\usepackage{epsfig}                      
\usepackage{epic,eepic}              
\usepackage{graphicx}                  
\usepackage{psfrag}                      
\usepackage{pgfplots}                  
\usepackage{tikz}                           

\usepackage{amsbsy}         
\usepackage{color}               

\usepackage{tabularx}             
\usepackage{booktabs}          
\usepackage{multirow}           
\usepackage{longtable}          
\usepackage{array}

\newcommand{\M}{\ensuremath{\mathbb{M}}}
\newcommand{\R}{\ensuremath{\mathbb{R}}}
\newcommand{\Z}{\ensuremath{\mathbb{Z}}}

\newcommand{\N}{\ensuremath{\mathbb{N}}}
\newcommand{\Q}{\ensuremath{\mathbb{Q}}}

\newcommand{\dis}{\displaystyle}
\newcommand{\vs}{\vspace{0,5cm}}
\newcommand{\sgn}{\mbox{\normalfont sgn}}
\newcommand{\T}{\mathbb{T}^2}
\newcommand{\Ss}{\mathbb{S}^2}

\newtheorem {theorem} {Theorem} 

\newtheorem {corollary}  {Corollary}
\newtheorem {lemma} {Lemma}
\newtheorem {definition}  {Definition}
\newtheorem {remark} {Remark}
\newtheorem {example} {Example}

\begin{document}


\title{Chaos in Piecewise Smooth Vector Fields on Two Dimensional Torus and Sphere}

\author{
  Ricardo M. Martins and   Durval J. Tonon}


%
%
%
%



\date{}
%

\maketitle


\begin{abstract}
In this paper we study the global dynamics of piecewise smooth vector fields defined in the two dimensional torus and sphere. We provide conditions under these families exhibits periodic and dense trajectories and we describe some global bifurcations. We also study its mi\-ni\-mal sets and characterize the chaotic behavior of the piecewise smooth vector fields defined in torus and sphere.
\end{abstract}

%
%
%
%
%
%
%
%
%


\section{Introduction}\label{secao introducao}

The theory of  piecewise smooth vector fields (PSVF for short) has been developing very fast in the last years, mainly due to its strong relation with branches of applied science, such as mechanical, aerospace engineering, electrical and electronic engineering, physics, economics, among others areas. Indeed, PSVF are in the boundary between mathematics, physics and engineering, see \cite{Makarenkov-Lamb} and \cite{Marco-enciclopedia} for a recent survey on this subject.

The PSVF are described by piecewise systems of differential equations, so that we have a smooth system defined in regions of the phase portrait. For some references on this subject, we suggest \cite{diBernardo-livro, F, Or}. 

The common frontier between the regions that separate the smooth vector fields is called switching manifold (or discontinuity manifold). There are a lot of research being made about the local behavior of PSVF near the switching manifold. For instance, see \cite{B-P-S, G-S-T, Kuz, Makarenkov-Lamb, T1,T3} for results about bifurcations on PSVF, \cite{J-T-T1, T1} for stability, \cite{J-T-T2} for reversibility and \cite{diBernardo-electrical-systems,5, C-B-F-J,6} for applications in relay and control systems. 

The mathematical formalization of the theory of PSVF was made precise by Filippov in \cite{F} and in this paper we consider his convention. As usual, we denote the switching manifold by $\Sigma$. The trajectory of the PSVF can be confined onto the switching manifold itself, when the trajectories on both sides of the switching manifold slides over the manifold after the meeting, and there remains until reaching the boundary of this (open) region, that is called sliding region.

The occurrence of such behavior, known as sliding motion, has been reported in a wide range of applications, especially on relay control systems and systems with dry friction, see \cite{diBernardo-livro, diBernardo-electrical-systems, 5,J-T-T1, J-T-T2,6, J-C} and \cite{Makarenkov-Lamb}. There are three distinguished types of regions on the discontinuity manifold: crossing region, attractive sliding region and repulsive sliding region. These regions fully describes what can happen on such manifold.

One of the biggest limitations of the theory of PSVF is the lack of global results, see \cite{Marco-enciclopedia}. In \cite{B-P-S}, authors considered conditions to get structural stability for PSVF defined in a compact differentiable manifolds, like a version of the {\it Peixoto's Theorem}. However, this paper does not consider a complete analysis of the dynamics in the cases presenting a non-trivial recurrence and tangential singularities on torus or sphere. In general, almost all results are local, in this way, the switching manifold can be taken as an hyperplane of codimension one. Recently a version of the {\it Poincar\'e-Bendixson Theorem} was proved to piecewise smooth case, see \cite{Buzzi-Carvalho-Eusebio} and therefore we are able to analyse global dynamics in compact manifolds.

Let $\M$ denote the two dimensional torus $\T$ or the sphere $\Ss$. In both cases, we will consider $\M$ obtained as the usual quotient of the square $I\times I=[0,1]\times [0,1]$.

Consider the switching manifold, denoted by $\Sigma$, as  $\Sigma=\Sigma_1\cup\Sigma_2=\{(t,0),\, t\in I\}\cup \{(t,\frac12),\, t\in I\}$ for $\M=\T$ and $\Sigma=\{(t,\frac12),\, t\in I\}$ for $\M=\Ss$. With the usual topologies on $I\times I$ generating $\M$, these choices of $\Sigma$ breaks $\M$ in two connected components, $\M^+=\M\cap (I\times [1/2,1])$ and $\M^-=\M\cap (I\times [0,1/2])$.

Rougly speaking, in the usual embedded model of the torus and sphere in $\R^3$, in the case $\M=\Ss$ we take $\Sigma$ separating $\M$ in top and bottom hemispheres, and in the case $\M=\T$ we take $\Sigma$ as the disjoint union of the inner and outer circles, separating the torus on top/bottom two half torus.

In this paper we study PSVF, denoted by $X$, defined on $\M$, with switching manifold $\Sigma$: we consider a 
piecewise smooth vector field $X=(X^+,X^-)$, where $X^+$ is defined on $\M^+$, $X^-$ is defined on $\M^-$. Over $\M^+\cap\M^-$ we adopt the Filippov's convention. We provide more details of this in Section \ref{secao teoria basica}.

Our main goal is to describe the global dynamics of $X$ in the cases where the singularities of $X$ are generic or do not exist (regular case). The main techniques that we use are the theory of contact between a smooth vector fields with the switching manifold and the dynamics of the first return map, which may be generalized for higher dimensions.

In \cite{G-S-T} and \cite{Kuz} are exhibit local normal forms for the two dimensional case. Therefore, we start the study of global dynamics on the two dimensional torus and sphere considering these families (only codimension zero). However, in \cite{Tiago-Durval} are exhibit the normal forms of codimension zero for three dimensional case, where the approach developed in this paper can be adapted.

The main results of this paper characterizes the global dynamics of PSVF not locally defined (out of euclidean spaces). We study global bifurcations, existence of minimal sets and chaos for families of PSVF, see Subsection \ref{subsecao-familiaPSVF}.

In the following we summarize the main results of this paper.

\vs

\noindent {\bf Theorem A:}  Let $X$ be a PSVF defined on $\M$ with switching manifold $\Sigma$. Suppose that $X$ is regular-transversal to $\Sigma$, without singularities or tangencies, taken in its normal form $X=(X^+,X^-)$, where $X^+(x,y)=(a,\sigma_1)$ and $X^-(x,y)=(b,\sigma_2)$, with $\sigma_i=\pm 1,i=1,2$. 

\noindent$(a)$ If $\sigma_1\sigma_2>0$ then:
\begin{itemize}
\item [$(i)$] for $\M=\T$ the trajectories of $X$ are periodic if $a\pm b\in\mathbb Q$; otherwise, $\M$ is a non trivial wandering set.
\item [$(ii)$] for $\M=\Ss$ the trajectories of $X$ connects the north and south poles of the sphere.
\end{itemize}
$(b)$ If $\sigma_1\sigma_2<0$ then the switching manifold presents a sliding motion that is a global attractor or repeller, according to the sigs of $\sigma_1$ and $\sigma_2$.\\

Theorem A is a version of {\it Kronecker-Weyl Equidistribution Theorem} for PSVF (see \cite{katok}, Proposition 4.2.1). Nevertheless, our approach it is distinctly from that used in the prove of this theorem. In our case, we analyse directly the dynamics of the first return map in the context of PSVF. See Subsection \ref{Subsecao-regular-regular} for the proof of Theorem A.

\vs

In the case where the PSVF present a finite number of fold singularities  or a finite number of critical points of the sli\-ding vector field, we show that the topological behavior of $X$ changes drastically when the number of these points changes.\\

\noindent {\bf Theorem B:} Let $X=(X^{+},X^-)$ be a PSVF defined on $\mathbb M$, where $X^+$ is a linear flow transversal to $\Sigma$, and $X^{-}$ is a vector field without singularities on $\M^-$ and with a finite and even number of visible tangencies.

$(a)$ If $\M=\T$, generically there exists a point $p^*\in \Sigma$ such that every trajectory has $p^*$ as its $\omega$-limit.

$(b)$ If $\M=\Ss$, generically there exists a point $p^*\in \Sigma$ and we can decompose $\Ss=\M_c\cup \M_h$, where $\M_c,\M_h$ are invariant sets with the property that every trajectory of points in $\M_c$ has $p^*$ as its $\omega$-limit and every trajectory of points in $\M_c$ is homoclinic to the south pole $p_S$.\\

Theorem B is a consequence of Theorems \ref{prop-Dobra-Regular} and \ref{prop-Hiperbolico} and Subsections \ref{Subsecao-Caos-Toro} and \ref{Subsecao-Caos-Esfera}. The hypothesis of an even number of tangencies is to avoid the presence of a line of singularities, and will be better explained in Subsection \ref{Subsecao-Dobra-regular}.\\

Recently some definitions of chaos for PSVF were introduced in \cite{BCE-ETDS} for planar piecewise smooth vector fields. We analyse the global dynamics of the PSVF on $\M$ and prove, under certain conditions, the occurrence of chaos:\\

\noindent {\bf Theorem C:} Let $X=(X^+,X^-)$ be a PSVF on $\M$ with the hypotheses of Theorem B.

$(a)$ If $\M=\T$, generically $X$ is chaotic on $\T$, that is, $X$ is topologically transitive and present sensitive dependence on the initial conditions.

$(b)$ If $\M=\Ss$, generically $\M=\M_h\cup \M_c$, where $\M_h,\M_c$ are open invariant sets, $\M_h$ is foliated by homoclinic trajectories and $X$ restricted to $\M_c$ is has a {\it kind} of chaotic, with $\overline \M_h\cap\overline \M_c$ a homoclinic trajectory.\\

Theorem C and its consequences (including more details about the chaotic behavior in (b)) are discussed in Section \ref{Secao-Caos}.\\

The organization of this paper is as follows: In Section \ref{secao teoria basica} we formalize some basic concepts on PSVF, as the first return map in this scenario. In Section \ref{secao resultados principais} the problem is described and some main results are stated and proved. The concept of chaos and minimal sets in the context of PSVF and the results obtained for the torus and sphere are given in Section \ref{Secao-Caos} and in Section \ref{secao-remarks} we finalize the paper with some comments and future directions of research.



\section{Basic Theory about PSVF}\label{secao teoria basica}

\subsection{Filippov's convention}\label{subsecao-Filippov}
First of all, we remember the usual construction of the two dimensional torus $\T$ and sphere $\Ss$. The model of the torus is provided by the following equivalence relation in $Q=[0,1]\times [0,1]\subset \R^2$:
\begin{equation}\label{relacao-equivalencia}(x,y)\sim (z,w) \Leftrightarrow x-z\in \Z, y-w\in \Z.\end{equation}
The model of the sphere $\Ss$ is the quotient $Q/\sim$, where $(0,s)\sim (1,s)$ for all $s\in[0,1]$, $(t,0)\sim (t',0)$ for every $t,t'\in [0,1]$ and $(t,1)\sim (t',1)$ for every $t,t'\in [0,1]$. In other words, we identify the vertical sides of the square $Q$ obtaining a cylinder and then collapse each of the borders of the cylinder to produce the north and south poles of the sphere, denoted by $p_N$ and $p_S$, respectively. 

\begin{remark}\label{observacao-const-esfera}In the above construction, the points $p_N$ and $p_S$ obtained from the collapse of the borders of the cylinder are not mapped in any parametrization (compatible with this model) of the sphere. With this model, we shall consider these points as singularities of every vector field defined on $\Ss$.
\end{remark}

So, we denote the space that the PSVF is defined by $\M$ where $\M=\T$ or $\M=\Ss$.

Consider $\Sigma_1 = \{ (x,y)\in Q ; y=0\}$ and $\Sigma_2 = \{ (x,y)\in Q ; y=\frac{1}{2}\}$. We denote $h_1(x,y)=y$ and $h_2(x,y)=y-\frac{1}{2}$. Therefore, $\Sigma_1=h_1^{-1}(0)$ and $\Sigma_2=h_2^{-1}(0)$. Clearly the switching manifolds $\Sigma_1$ and $\Sigma_2$ are the separating boundaries of the regions
$\Sigma^-=\{(x,y)\in M; 0\leq y \leq \frac{1}{2}\}$ and $\Sigma^+=\{(x,y)\in M; \frac{1}{2}\leq y \leq 1\}$.

Designate by $\mathfrak{X}^r$ the space of $C^r$-vector fields on $\M$ endowed with the $C^r$-topology with $r=\infty$ or $r\geq 1$ large enough for our purposes. Call \textbf{$\Omega^r$} the space of vector fields $X: \M \rightarrow \M$ such that
\[
X(x,y)=\left\{\begin{array}{ll}  X^+(x,y),    &\mbox{ for } \quad (x,y) \in \Sigma^+,\\
X^-(x,y),   & \mbox{ for }  \quad (x,y) \in \Sigma^-,
\end{array}\right.
\]
where $X^+=(X_1^+,X_2^+)$ and $X^- = (X_1^-,X_2^-)$ are in $\mathfrak{X}^r$. Let $h\in\{h_1,h_2\}$ and $\Sigma\in\{\Sigma_1,\Sigma_2\}$. We denote $X^{\pm}h(p)=\langle X^{\pm}(p),\nabla h(p)\rangle$ and $(X^{\pm})^{n}h(p)=\langle X^{\pm}(p), $ $\nabla (X^{\pm})^{n-1}h(p)\rangle$ the Lie's derivatives, where $\langle \cdot,\cdot\rangle$ denote the canonical inner product. We may consider $\Omega^r = \mathfrak{X}^r \times \mathfrak{X}^r$ endowed with the product topology and denote any element in $\Omega^r$ by $X=(X^+,X^-),$ which we will accept to be multivalued in points of $\Sigma$. The basic results of differential equations, in this context, were stated by Filippov in \cite{F}. Related theories can be found in \cite{diBernardo-livro,Or,Marco-enciclopedia} and references therein.

On $\Sigma\in\{\Sigma_1,\Sigma_2\}$ we generically distinguish three regions:\\
\noindent {{\bf crossing region:}} $\Sigma^c=\{ p \in \Sigma;\, X_2^+(p)  X_2^-(p)> 0 \}$,\\ 
\noindent {{\bf stable sliding region:}} $\Sigma^{s}= \{ p \in \Sigma; X_2^+(p)<0, X_2^-(p)>0 \}$ and \\
\noindent {{\bf unstable sliding region:}} $\Sigma^{u}= \{ p \in \Sigma; X_2^+(p)>0, X_2^-(p)<0\}$.\\

When $q \in \Sigma^s$, following the Filippov's convention, the {sliding vector field} associated to $X\in \Omega^r$ is the vector field $\widehat{X}^s$ tangent to $\Sigma^s$, expressed in coordinates as 
\[\widehat{X}^s(q)= \frac{1}{(X_2^- - X_2^+)(q)} ((X_1^+- X_1^-)(q),0),\]which, by a time rescaling, is topologically equivalent to the \textbf{normalized sliding vector field}
\begin{equation}\label{expressao-campo-deslizante}
X^{s}(q)=(X_1^+- X_1^-)(q).
\end{equation}

Note that $X^s$ can be C$^r$-extended to the closure $\overline{\Sigma^s}$ of $\Sigma^s$. The points $q \in \Sigma$ such that $X^s(q)=0$ are called \textbf{pseudo equilibrium} of $X$ and the points $p \in \Sigma$ such that $X^+h(p)X^-h(p)=0$ are called \textbf{tangential singularities} of $X$ (i.e., the trajectory through $p$ is tangent to $\Sigma$). We say that $q\in\Sigma$ is a \textit{regular point} if $q\in \Sigma^c$ or $q\in \Sigma^s$ and $X^s(q)\neq0$.

A tangential singularity $q\in\Sigma$ of $X^-$ is a \textit{fold point} of $X^-$ if $X^-h(q)=0$ but $(X^-)^{2}h(q)\neq0$, visible tangency if $(X^-)^{2}h(q)<0$ and invisible tangency if $(X^-)^{2}h(q)>0$.

\begin{definition}
	The flow $\phi_X$ of $X\in \Omega^r$ is obtained by the concatenation of flows of $X^+,X^-$ and $X^s$, denoted by $\phi_{X^+}, \phi_{X^-}$ and $\phi_{X^s}$, respectively.
\end{definition}

In the following we define non-wandering point of the flow of $X$.

\begin{definition}
	A point $p$ is called non-wandering if for any neighborhood $U$ of $p$ and $T > 0$, there exists some $|t| > T$ such that $\phi_X(t,U)\cap U \neq  \emptyset$.
\end{definition}


\subsection{Families of PSVF}\label{subsecao-familiaPSVF}

To start the analysis of PSVF in the two dimensional torus and sphere, we consider the following families of systems that are locally, in two dimensional euclidean spaces, in a subset of codimension zero in $\Omega^r$ (generic). These families are given in \cite{G-S-T} and \cite{Kuz}. Unfortunately, there is not in the literature an result that provides a generic family of PSVF for compact manifolds.  However, as these families are locally generic in $\R^2$ then they have to be preserved in the study of global dynamics in $\T$ and $\Ss$.

\begin{definition}\label{definicao fold point}
	Let $X=(X^+,X^-) \in \Omega^r$. We say that  $p \in \Sigma$ is a
	\begin{itemize}
		\item [$(a)$] \textbf{fold-regular singularity} of $X$ if $p$ is a fold point of $X^-$ and $X^+(p)$ is transversal to $\Sigma$ at $p$.
		
		\item [$(b)$] \textbf{hyperbolical pseudo equilibrium} of $X^s$ if $p\in \Sigma^s\cup \Sigma^u,X^{s}(p)=0$ and $(X^{s})'(p)\neq 0$.
	\end{itemize}\end{definition}
	
	Based in \cite{G-S-T} and \cite{Kuz} we get the list of codimension zero PSVF in $\Omega^r$:\\
	
	$\bullet$ $\Omega_0(c)= \{X;\, p\in \Sigma \textrm{ is a regular
		point and } p\in \Sigma^c\}$;
	
	$\bullet$ $\Omega_0(s)= \{X;\, p\in \Sigma \textrm{ is a regular
		point and } p\in \Sigma^s\}$;
	
	$\bullet$ $\Omega_0(u)= \{X;\, p\in \Sigma \textrm{ is a regular
		point and } p\in \Sigma^u\}$;
	
	$\bullet$ $\Omega_0(fr)= \{X;\, p\in \Sigma \textrm{ is fold-regular
		singularity}\}$;
	
	$\bullet$ $\Omega_0(h)=\{X;\, p \in \Sigma \textrm{ is a hyperbolic
		pseudo equilibrium of }X^{s}\}$.


\section{Main Results for PSVF}\label{secao resultados principais}

\subsection{Regular-regular case}\label{Subsecao-regular-regular}
Based on \cite{G-S-T} and \cite{Kuz}, we consider the following expression of $X=(X^+,X^-)$ in this case
\[
X^+(x,y)=(a,\sigma_1), X^-(x,y)=(b,\sigma_2),
\]
where $\sigma_1, \sigma_2=\pm1$. The proof of Theorem A is divided in two cases:

\noindent $\bullet$ $\Sigma=\Sigma^c$, that is, $X\in \Omega_{0}(c)$;\\
\noindent $\bullet$ $\Sigma=\Sigma^s\cup \Sigma^u$, that is, $X\in \Omega_{0}(s)\cup \Omega_0(u)$.

\subsubsection{$X\in \Omega_0(c)$}

First of all, we have to define the first return map in this case. We suppose that $X_2^+(p)>0$ and $X_2^-(p)>0$. The construction for the other case, $X_2^+(p)<0$ and $X_2^-(p)<0$, is analogous.

Consider $p\in \Sigma_1$ and $t(p)$ the first positive time such that $\phi_{X^-}(t(p),p)=p_1\in\Sigma_2$. Let be $t(p_1)$ the positive time such that $\phi_{X^+}(t(p_1),p_1)=p_2\in\Sigma_1$. Therefore, the first return map is given by this composition:
\[\varphi_X(p) = \phi_{X^+}(t(p_1),\phi_{X^-}(t(p),p).\]

\begin{theorem}\label{caso-regular}
$(a)$ If $\M=\T$ then the orbits of $X$ are periodic if and only if $a\pm b\in \Q$. Besides, if $a\pm b\notin \Q$ then the torus $\T$ is a non trivial wandering set. 

$(b)$ If $\M=\Ss$ then $\Sigma$ is a crossing region, and all the trajectories of $X$ on $\Ss$ are lines connecting $p_N$ and $p_S$. In particular, the flow of $X$ is $C^0$-equivalent to the flow of a smooth vector field with just two singularities: one attractor and other repeller.
\end{theorem}

\begin{proof}
We initially prove item $(a)$. Note that the existence of periodic orbits is equivalent to the existence of $n_0\in \N$ such that $\varphi_X^{n_0}(p)=kp$, where $k\in \N$ and $p\in\Sigma_1$.


Given $p=(x,0)\in \Sigma_1$ we get $\phi_{X^-}(p)=(a/2\sigma_1+x,1/2)=(x_2,1/2)$ and $\phi_{X^+}(x_2,1/2)=(a/2\sigma_1+b/2\sigma_2+x,1)=(x_3,1)$.

In this case, the global dynamics of $X$ is given by the first return map and we can exhibit explicitly:
\begin{equation}\label{equacao-fluxo-RR-costura}
\varphi_X^{n}(x,0)=\overline{\left( \frac{n}{2}\left(\frac{a}{2\sigma_1}+\frac{b}{2\sigma_2}\right) +x, \frac{n}{2} \right)},  
\end{equation}for $n\geq 2$, where the overline denotes the representant of the class of equi\-va\-len\-ce, defined in \eqref{relacao-equivalencia}.

Therefore, the necessary and sufficient condition to get periodic orbits, in this case, is $\overline{\varphi_X^{n_0}(x,0)}=\overline{(x,0)}$, for some $n_0\in\N$. By \eqref{equacao-fluxo-RR-costura}, the last equation is satisfied if and only if 
\[
\frac{n_0}{4}\left(\dis\frac{a\sigma_2+b\sigma_1}{\sigma_1\sigma_2}\right)\in \Z,
\]
or equivalently,
\[
a\pm b\in \Q.
\]
So, we conclude that there exist periodic orbits for $X$ if and only if $a\pm b\in \Q$. Otherwise, if $a\pm b\notin \Q$ then the orbits of $X$ are not periodic and $X$ does not have singular points. Therefore, by applying the version of {\it Poincar\'e-Bendixson Theorem} for PSVF in \cite{Buzzi-Carvalho-Eusebio}, we conclude that the orbits of $X$, in this case, are dense. This ends the proof of (a).

The proof of Item $(b)$ follows by the construction of model of the sphere where $p_N$ and $p_S$ are singularities, see Remark \ref{observacao-const-esfera}.
\end{proof}

%

\begin{example}[Limit cycles]\label{ex2.2}
Note that in Theorem \ref{caso-regular}, if we do not take $X$ in normal form, just regular-regular, the result is no longer valid. In particular, we can have cycle limits for PSVF on $\M=\T$: consider the vector field
\[X_\alpha(x,y)=\left\{
\begin{array}{lcl}
(\cos(2\pi x),1)&,&(x,y)\in [0,1]\times \left[\frac12,1\right],\\
(1,\alpha)&,&(x,y)\in [0,1]\times \left[0,\frac12\right],
\end{array}
\right.
\]
with $\alpha>0$ to be defined below.
Note that the lines $\gamma_1(t)=(\frac14,t)$, $t\in [\frac12,1]$ and $\gamma_2(t)=(\frac34,t)$, $t\in [\frac12,1]$ are invariants.

Let $\epsilon>0$ small and consider the point $p=(\frac34-\epsilon,\frac12)$. Take the integral curve of the flow through $p$ and let $q=(x_1,1)$ be its intersection with the segment $[0,1]\times \{1\}$.  Let $\tilde q=(x_1,0)$ and $v=p-\tilde q=(v_1,v_2)$. By our construction, $v_1\neq 0$ and $v_2\neq 0$. Take $\alpha=\dfrac{v_2}{v_1}$. The trajectory $\gamma$ of $\tilde q$ by the flow of $X_\alpha$, with this choice of $\alpha$, is closed.

Now take a point $z\in [0,1]\times\{0\}$ between $(\frac14,0)$ and $\tilde q$. It is easy to see that the trajectory of $z$ by the flow of $X_\alpha$ is not closed and approaches $\gamma$. Then the same holds with the trajectories of points in the left side of $\tilde q$, but in a small neighborhood. 

If we take a point in the right side of $\tilde q$, the same property still holds. Then $\gamma$ is a limit cycle. See Figure \ref{fig-cl} for a representation of this limit cycle.
\begin{figure}[h]
\begin{center}
\begin{overpic}[width=3.5cm,bb=0 0 176 212]{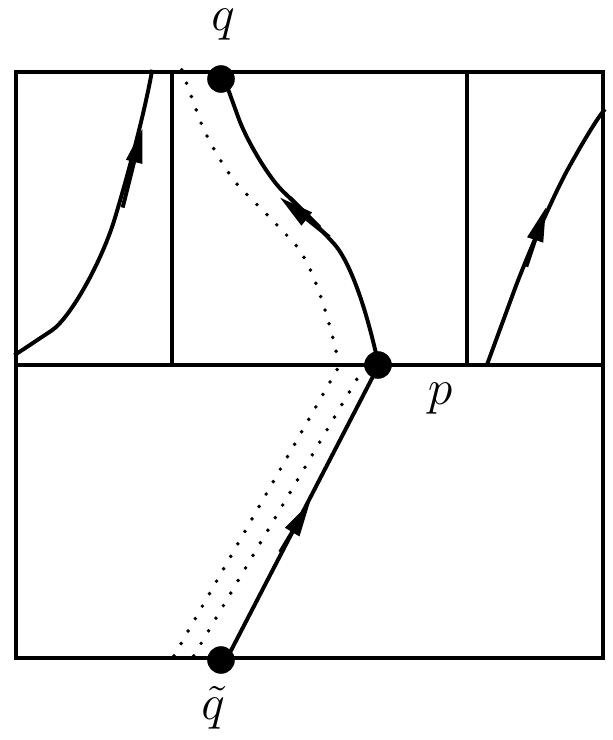}
\put(15,26){$\gamma$} 
\end{overpic}
\end{center}
\caption{Limit cycle of system presented in Example \ref{ex2.2}.}\label{fig-cl}
\end{figure}
\end{example}

\subsubsection{$X\in \Omega_0(s)\cup  \Omega_0(u)$}

Note that as $X\in \Omega_0(s)\cup  \Omega_0(u)$ then the sign of $\sigma_1$ determines the sign of $\sigma_2$.

\begin{theorem}\label{caso-regular-deslize}
Consider $\M=\T$ or $\M= \Ss$. If $\sigma_1<0$ (resp. $\sigma_1>0$) then $\Sigma_2$ is a stable sliding region (resp. unstable sliding region). In both cases $\Sigma_2$ is itself a closed trajectory, solution of the constant differential equation
\[\left\{
\begin{array}{lcl}
\dot x&=&(\sigma_2 a-\sigma_1 b)/(\sigma_2-\sigma_1),\\
\dot y&=&0.
\end{array}\right.
\]
Besides then:

$(a)$ if $\M=\T$ and 
\begin{itemize}
\item [$(i)$] $\sigma_1>0$ then $\Sigma_1$ (resp. $\Sigma_2$) is a global attractor (resp. global repeller) for $X$. 
\item [$(ii)$] $\sigma_1<0$ then $\Sigma_1$ (resp. $\Sigma_2$) is a global repeller (resp. global attractor) for $X$. 
\end{itemize}

$(b)$ if $\M=\Ss$ and $\sigma_1<0$ (resp. $\sigma_1>0$) then the flow of $X$ is $C^0$-equivalent to the flow of a vector field on $\Ss$ with two repelling singularities and one attracting closed orbit (resp. two attracting singularities and one repelling closed orbit).
\end{theorem}

\begin{proof}
We consider firstly $\sigma_1<0$. In this case $\Sigma_{1},\Sigma_{2}$ coincides with $\Sigma^{u},\Sigma^{s}$, respectively.  In $\Sigma_{2}$, the normalized sliding vector fields $X^{s}$, given in \eqref{expressao-campo-deslizante}, is expressed explicitly as
\[\left\{
\begin{array}{lcl}
\dot x&=&(\sigma_2 a-\sigma_1 b)/(\sigma_2-\sigma_1),\\
\dot y&=&0.
\end{array}\right.
\]
Note that $X^{s}$ is regular and when its trajectory reaches the point $(1,1/2)$, by equi\-va\-len\-ce class, is identified with $(0,1/2)$ and follows again the flow. Therefore, we say that the orbit of $X^{s}$ coincides with the region $\Sigma_{2}$. In $\Sigma_{1}$, when $\M=\T$, we have similar behavior and we conclude that $\Sigma_{1}$ and $\Sigma_{2}$ are periodic orbits. Besides, $\Sigma_{1},\Sigma_{2}$ is a global repeller for $X$ (unstable limit cycle) and global attractor for $X$ (stable limit cycle), respectively.

In case when $\M=\Ss$ then there exists singularities $p_N$ and $p_S$ whose stability is determined by the sing of $\sigma_1$. More precisely, if $\sigma_1>0, \sigma_1<0$ then $\Sigma_2$ is an unstable sliding region, stable sliding region, respectively. Therefore, if $\sigma_1>0, \sigma_1<0$ then $p_N, p_S$ is an attractor, repeller, respectively.
\end{proof}


\subsection{Fold-regular case}\label{Subsecao-Dobra-regular}

We suppose that $X^+$ is a regular vector fields and $X^-$ presents fold singularities in $\Sigma_1\cup \Sigma_2$.

\begin{lemma}\label{obs-conecta}
Generically it does not exist trajectories of $X^-$ connec\-ting two fold singularities in $\Sigma_1\cup \Sigma_2$. 
\end{lemma}

\begin{proof}
In fact, the trajectory $\gamma$ connecting two tangency points it is a distinguish trajectory in the sense that $\gamma$ most be preserved by a topological equivalence. However, this is a not robust behavior and it is necessary one parameter to get the unfolding of this dynamic. In Figure \ref{bifurcacao-dobra} we present a geometric approach to this unfolding. So, the presence of this trajectory $\gamma$ provides a global bifurcation of codimension one. Therefore, we assume that does not exist a trajectory connecting tangential singularities.
\begin{figure}[h]
	\begin{center}
		\begin{overpic}[width=9cm]{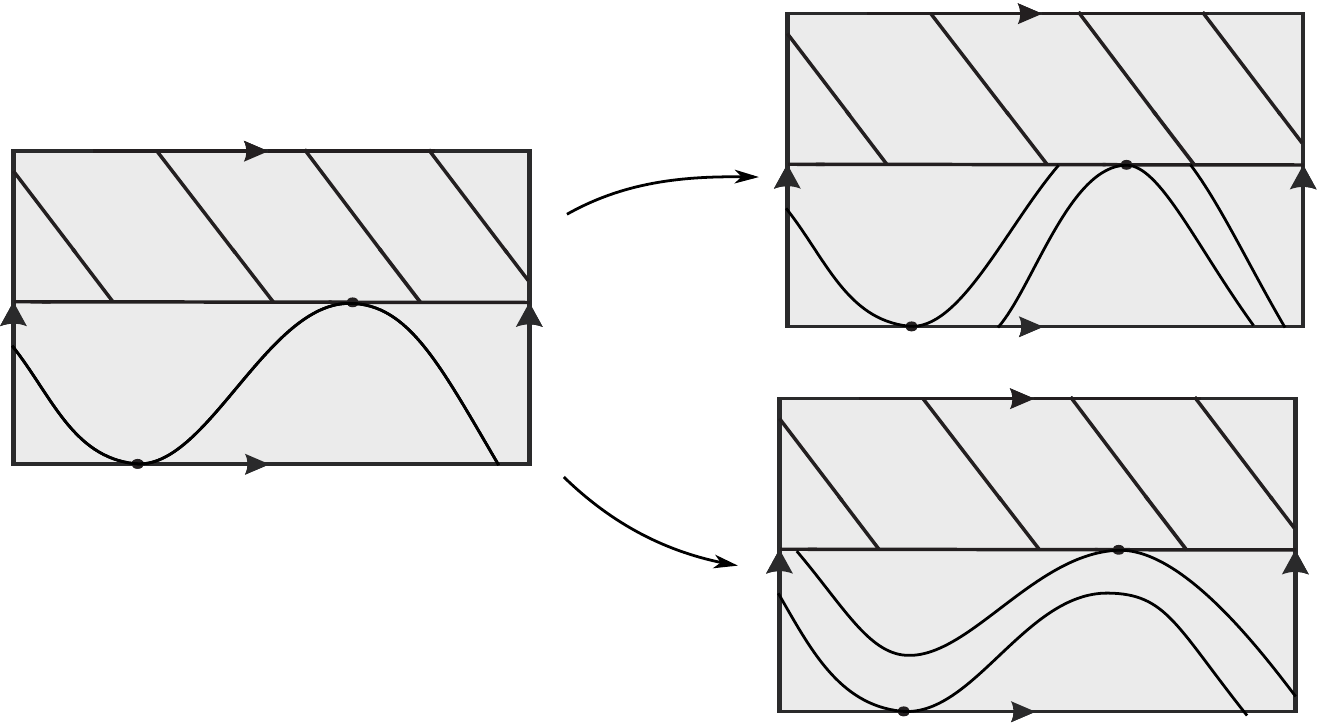}
			\put(15,26){$\gamma$} 
		\end{overpic}
	\end{center}
	\caption{The unfolding of a trajectory connecting two fold singularities.}\label{bifurcacao-dobra}
\end{figure}
\end{proof}

The next result provides a relation between the existence of invisible and visible fold singularities of $X^-$.

\begin{lemma}\label{corresp-fold}
Consider $\M=\T$. For each visible/invisible singularity of $X^{-}$ on $\Sigma_i$ there exists another invisible/visible singularity of $X^{-}$ on $\Sigma_j$ ($i\neq j$).
\end{lemma}

\begin{proof}
For each $p_i\in \Sigma_1$, a visible fold singularity of $X^-$, we consider $\Sigma_i^-, \Sigma_i^+$ transverse sections to $\Sigma_1$ at left, right, respectively, of $p_i$, see Figure \ref{dobra-visivel-invisivel}. Let be $q_i^{\pm}$ the intersection between the trajectory passing through $p_i$ with $\Sigma_i^{\pm}$. For each point $q_{i}^-(\eta)\in \Sigma_i^-$, above of $q_{i}^-$ we get the trajectory $\gamma_i(\eta)$ and a unique point $p_i(\eta)$ such that $X^-h_2(p_i(\eta))=0$. Besides, the points $p_{i}(\eta)$ varies continuously according $q_{i}^-(\eta)$. Therefore, we obtain a curve $\mathcal{C}_i(\eta)$ in an half torus (the region between $\Sigma_1$ and $\Sigma_2$) such that $X^-h_2(\mathcal{C}_i(\eta))=0$ and by construction, $\mathcal{C}_i(\eta)$ is transversal to $\Sigma_2$ at a point $\eta_i$, which is an invisible fold singularity of $X^-$. So, we conclude the proof. 
\end{proof}

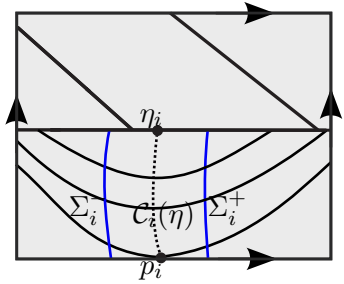
\begin{figure}[h]
\begin{center}
\begin{overpic}[width=4.5cm]{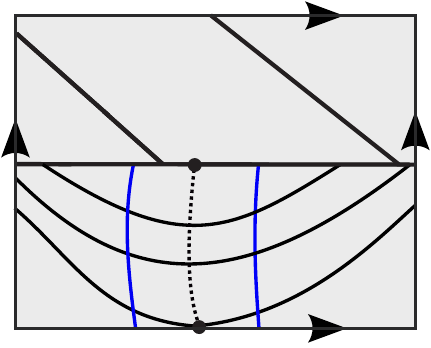}
\put(19,17){$\Sigma_i^-$} \put(60,17){$\Sigma_i^+$} \put(38,15){$\mathcal{C}_i(\eta)$} \put(40,0){$p_i$} \put(40,45){$\eta_i$}
\end{overpic}
\end{center}
\caption{The construction of transverse sections $\Sigma_i^{\pm}$ and curve $\mathcal{C}_i(\eta)$ on $\T$.}\label{dobra-visivel-invisivel}
\end{figure}
In this case, the main result is the following:

\begin{theorem}\label{tns}
$(a)$ If $\M=\T,\Ss$ and the number of fold singularities of $X^-$ in $\Sigma_{2}$ is odd then

\begin{itemize}
\item [$(i)$] there exists a curve of singular points of $X^-$ given by $\{(0,y); y\in \R\}$ or 
\item [$(ii)$] $\langle X^-(0,y),\nabla h_{2}(0,y) \rangle =0$ for all $y\in \R$. In other words, the straight line $\{(0,y); y\in \R\}$ is composed by tangential singularities of $X^-$ with $\Sigma_{2}$.
\end{itemize}

$(b)$  If $\M=\T,\Ss$ and the number of fold singularities of $X^-$ in $\Sigma_{2}$ is $2n$ then $\Sigma_1$ and $\Sigma_2$ breaks into $2n$ regions each. According to the orientations, the regions in $\Sigma_1,\Sigma_2$ change between crossing and stable sliding region or crossing and unstable sliding region.
\label{prop-Dobra-Regular}\end{theorem}

\begin{proof}
Item $(a)$: We detail the proof for the case when $\M=\T$. The analysis for $\M=\Ss$ is analogous and we omit. Consider the following orientation given in Figure \ref{dobra-impar-toro}, the analysis for the other cases is similar. The contact, at $(x,y)$, between the smooth vector field $X^-$ with $\Sigma_{2}$ is given by $\langle X^-(x,y), \nabla h_{2}(x,y)\rangle $.  

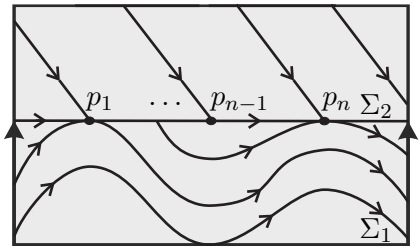
\begin{figure}[h]
\begin{center}
\begin{overpic}[width=5.5cm,bb=0 0 156 91]{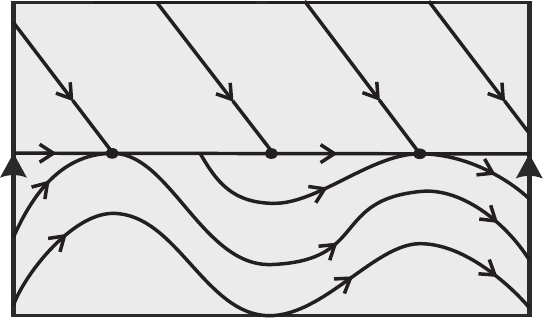}
\put(86,2){$\Sigma_1$} \put(86,32){$\Sigma_2$} \put(20,34){$p_1$}\put(77,34){$p_{n}$}\put(50,34){$p_{n-1}$}\put(35,34){$\dots$}
\end{overpic}
\end{center}
\caption{The dynamics for the fold-regular case when there exist an odd number of fold points of $X^-$ in $\Sigma_{2}$.}\label{dobra-impar-toro}
\end{figure}

Note that, with this orientation, we get $\langle X^-(\eta,y), \nabla h_{2}(\eta,y)\rangle >0$ and $\langle X^-(1-\eta,y), \nabla h_{2}(1-\eta,y)\rangle<0$, where $\eta$ is small and non-negative real number. Therefore, 
\[
\dis\lim_{\eta \to 0} \langle X^-(\eta,y), \nabla h_{2}(\eta,y)\rangle = \dis\lim_{\eta \to 0} \langle X^-(1-\eta,y), \nabla h_{2}(1-\eta,y)\rangle=0.
\] As $\nabla h_2(0,y)=(0,1)$ we conclude that $X^-(0,y)= 0$ or $\langle X^-(0,y), \nabla h_{2}(0,y)\rangle$ $=0$.

Item $(b)$: Consider, without loss of generality, that $X^+h_2(p)<0$ for all $p\in \Sigma_2$. Let be $p_i,i=1, \dots, 2n$ the fold singularities of $X^-$ in $\Sigma_2$. So $X^-h_2(p_i)=0,i=1, \dots,2n$ and the sign of $(X^-)^2h_2(p)$ changes when $p\in \Sigma_2$ pass to each fold singularity $p_i$. Therefore, we get that each time which $p$ pass through $p_i$ we obtain a change of regions on $\Sigma_2$, between $\Sigma^c$ and $\Sigma^s$. 


\end{proof}

\begin{remark}By Lemma \ref{corresp-fold} and Theorem \ref{prop-Dobra-Regular}, considering $\M=\T$, if the number of visible folds is even, then the number of invisible folds is also even. So, if $X^-$ has no critical points or it is not flat on $\Sigma_2$, the minimum number of folds is four: two visible and two invisible.
\end{remark}

The next result is an important consequence of Theorem \ref{prop-Dobra-Regular}.

\begin{corollary}\label{not-ss} If $X$ is a fold--regular PSVF on $\M$ with an odd number of fold singularities, then $X$ is not structural stable. Moreover, in this case the codimension of the bifurcation diagram is equal to infinity.
\end{corollary}
\begin{proof}Follows as consequence of the presence of a curve of tangential or critical points of the smooth vector field $X^-$. 
\end{proof}

The kind of change of structure present in Corollary \ref{not-ss} does not happen in PSVF defined in euclidean spaces. In fact, in this context, the family of $X$ presenting a fold-regular singularities is generic (codimension zero), see \cite{G-S-T, Kuz}. So, considering PSVF defined in $\T$ or $\Ss$, when the number of fold singularities changes (even to odd), the behavior changes between a codimension zero to codimension infinity!


\subsection{Critical point of $X^{s}$}  

In this case, we suppose that the number of pseudo equilibrium of $X^{s}$ is finite and all of them distinct and hyperbolic, i.e., $X\in \Omega_0(h)$. Let $p_{i},i=1, \dots, n$ these pseudo equilibriums. 

Let us define the sign function by $\sgn(0)=0$, $\sgn(x)=1$ if $x>0$, $\sgn(x)=-1$ if $x<0$. We obtain the following results:

\begin{theorem}If $\M=\T$ or $\M=\Ss$, the number of pseudo equilibrium of $X^{s}$ is even and if $p_i$ is an attracting (resp. repelling) pseudo equilibrium of $X^s$ then $p_{i+1}$ is a repelling (resp. attracting) pseudo equilibrium of $X^s$.
\label{prop-Hiperbolico}\end{theorem}

\begin{proof}
As $p_{i}$ is hyperbolic, we define $I(p_{i})=sgn((X^{s}(p_{i})'))$ as the index of the hyperbolic pseudo equilibrium $p_i$ of $X^s$. We suppose, without loss of generality, that $I(p_{i})=1$, i.e., $p_{i}$ is a hyperbolic repeller pseudo equilibrium of $X^{s}$. In this way, as $I(p_{i-1}) \neq 0$ and $I(p_{i+1})\neq 0$ the choose of sign of $I(p_{i})$ (the stability of $p_{i}$) also determinates that the signs of $I(p_{i-1})$ and $I(p_{i+1})$ are both negative, i.e., $p_{i-1}$ and $p_{i+1}$ are hyperbolic attractors (pseudo equilibria of $X^{s}$). 

Therefore, the stability of all the others pseudo equilibrium are determinate. In fact, by the same argument we get that $p_{i-2}$ and $p_{i+2}$ are hyperbolic repeller pseudo equilibrium, so $p_{i-3}$ and $p_{i+3}$ are hyperbolic attractor pseudo equilibrium and we repeat the same analysis for the other pseudo equilibrium.

Now we prove that the number of pseudo equilibria of $X^s$ is even. 

In fact, if the number of pseudo equilibria is odd, at the end of this process remain two consecutive pseudo equilibria. Considering the previous analysis of stability of each pseudo equilibrium,  we get that these two consecutive points have the same stability, what is a contradiction with the fact that these one are hyperbolic. So, we conclude that at the end of the analysis remain only one pseudo equilibrium of $X^{s}$. Therefore, we get that the number of pseudo equilibria of $X^{s}$ is even and the stability of each one is alternating between attractor and repeller. 
\end{proof}

As consequence of Theorem \ref{prop-Hiperbolico} we get that

\begin{corollary}
If the number of pseudo equilibria of $X^{s}$ is odd then there exists at least one $p_{i}^{*}$ that is a saddle-node pseudo equilibrium of $X^{s}$, i.e., in one side of $p_{i}^{*}$, in $\Sigma^{s}$, the orbits of $X^{s}$ arrive at $p_{i}^{*}$ and in another side they leave. 
\label{coro-sela-no}\end{corollary}

\begin{remark}
As consequence, in this case, the codimension of the PSVF that presents an odd number of pseudo equilibria of $X^{s}$ is, at least, one.
\end{remark}

\begin{example}
Consider the $1$-parameter family of vector fields define on torus $\T$:
\[X_\alpha(x,y)=\left\{
\begin{array}{lcl}
X^+(x,y)=(-1,-\frac92 +24x-24x^2)&,&(x,y)\in [0,1]\times \left(\frac12,1\right),\\
X_{\alpha}^-(x,y)=(-1,\alpha)&,&(x,y)\in [0,1]\times \left(0,\frac12\right).
\end{array}
\right.
\]
Note that the vector field $X^+$ has four fold points: visible folds at $(1,\frac14)$, $(\frac12,\frac34)$ and invisible folds at $(\frac12,\frac14)$,$(1,\frac34)$ (recall that $(1,\frac14)\sim(0,\frac14)$ and $(1,\frac34)\sim(0,\frac34)$).

For every $\alpha\neq 0$, $\Sigma_1$ and $\Sigma_2$ breaks into two regions each: \\
\noindent $(a)$ if $\alpha>0$, in $\Sigma_2$ the segment $\{(t,\frac12),\, t\in [0,\frac14)\cup(\frac34,1]\}$ is a crossing region and the segment $\{(t,\frac12),\, t\in (\frac14,\frac34)\}$ is an unstable sliding region, whereas in $\Sigma_1$ the segment $\{(t,1),\, t\in [0,\frac14)\cup(\frac34,1]\}$ is a crossing region and the segment $\{(t,1),\, t\in (\frac14,\frac34)\}$ is a stable sliding region.\\
\noindent $(b)$ if $\alpha<0$, in $\Sigma_2$ the segment $\{(t,\frac12),\, t\in [0,\frac14)\cup(\frac34,1]\}$ is a stable sliding region and the segment $\{(t,\frac12),\, t\in (\frac14,\frac34)\}$ is a crossing region, whereas in $\Sigma_1$ the segment $\{(t,1),\, t\in [0,\frac14)\cup(\frac34,1]\}$ is an unstable sliding region and the segment $\{(t,1),\, t\in (\frac14,\frac34)\}$ is a crossing region.
\end{example}


\section{Chaos and minimality of PSVF on torus and sphere}\label{Secao-Caos}

In this section, our objective is study the minimal sets of the PSVF presenting a finite number of fold singularities on the torus and sphere. We analyse the global dynamics and prove, under certain conditions, the existence of minimal sets in $\M=\T$ or $\M=\Ss$ and the occurrence of chaos.

\subsection{Some concepts and definitions under the chaos and minimal sets of PSVF}

In the following, we consider some definitions, which is stated in \cite{BCE-ETDS}, of minimal sets and chaos in this context.

\begin{definition}
A set $A\subset \M$ is invariant for $X$ if for each $p\in A$ and all global trajectory $\phi_{X}(t,p)$ passing through $p$ it holds $\phi_{X}(t,p)\in A$.
\end{definition}

\begin{definition}\label{def-minimal}
Let be $X\in \Omega^r$. A set $S \subset \M$ is minimal for $X$ if:
\begin{itemize}
\item[$(a)$] $S\neq \emptyset$;
\item[$(b)$] $S$ is compact;
\item[$(c)$] $S$ is invariant for $X$;
\item[$(d)$] $S$ does not contain proper subsets satisfying $(a)-(c)$.
\end{itemize}
\end{definition}

One of the first papers that treat chaos in the context of PSVF was \cite{J-C2} where the authors adapt the classical definition of chaos. In fact, this definition, stated in \cite{BCE-ETDS} for example,  consider topological transitivity and sensibility to initial conditions, which are defined in the following.

\begin{definition}
$X\in \Omega^r$ is topologically transitive on an invariant subset $A\subset \M$ if for every pair of non-empty open sets $U,V\subset \M$ there exist $p\in U$ and $t_0>0$ such that the positive global trajectory $\phi_X^+(t,p)$ evaluated at $t=t_0$ is in $V$, i. e., $\phi_X^+(t_0,p)\in V$. 
\end{definition}

\begin{definition}
$X\in \Omega^r$ presents sensitive dependence on a compact invariant set $A$ if there exist a fixed $r>0$ with $r<diam(A)$ such that for each $x\in A$ and $\varepsilon>0$ there exist $y\in B_{\varepsilon}(x)\cap A$ and positive global trajectories $\Gamma_x^+$ and $\Gamma_y^+$ passing through $x,y$, respectively, satisfying 
\[d_H(\Gamma_x^+, \Gamma_y^+)=\dis\sup_{a\in \Gamma_x^+, b\in \Gamma_y^+}d(a,b) >r,\]
where $diam(A)$ is a diameter of $A$ and $d(a,b)$ is a euclidean distance between $a$ and $b$.
\end{definition}

With these definitions we are able to define a chaotic set, analogously to the smooth case.

\begin{definition}
We say that $X$ is chaotic on a compact invariant set $A$ if is topologically transitive and present sensitive dependence on $A$.
\end{definition}

To prove the invariance of a compact subset of $\M=\T,\Ss$ we need to consider the dynamics of the half-first return map, that is defined in following. The purpose to consider the half-first return is that the complexity of the dynamics of $X$ is provided by the existence of a finite number of fold singularities, that occur in the bottom half torus.

\begin{definition}
Consider a transverse section $\Lambda=\{\overline{(0,\xi)};\xi\in [0,1/2]\}, q\in \Lambda$ and the trajectory $\phi_{X^-}(t,q)$ of $X^-$ passing through $q$. Let be $t(q)$ the first positive time such that $\phi_{X^-}(t(q),q)\in \Lambda$. We define the half-first return map $\pi: \Lambda\rightarrow \Lambda$ such that $\pi(q)=\phi_{X^-}(t(q),q)$ and the displacement map $d:\Lambda \rightarrow \Lambda$ such that $d(q)=\pi(q)-q$.
\end{definition}
Note that the half-first return map is a $C^r$-difeomorfism. In the one dimensional transverse section $\Lambda$ we consider the order relation: given $q_i=(0,\xi_i)\in \Lambda$ we say that $q_i< q_j$ if and only if $\xi_i<\xi_j$.

Let be $q_{i}=\Gamma_{p_{i}}\cap \Lambda$ the point where $\Gamma_{p_i}$ intercepts $\Lambda$ for the first time, $\tilde{q_i}=\pi(q_i)$ and $\tilde{\Gamma}_{p_i}$ the arc of trajectory of $X^-$ passing through $p_i$ whose frontier is $q_i$ and $\tilde{q_i}$.


\vs

\begin{definition}\label{pstar} Consider $p^*\in \Sigma_2$ a fold singularity of $X^-$ satisfying:

\begin{itemize}
\item[$(1)$]the trajectories of the sliding vector fields defined in the adjacent stable sliding region of $p^*$ (remember that $p^*$ is in frontier of stable sliding sliding region) converges to $p^*$;
\item[$(2)$]considering the region $R_1\subset \Sigma^-$ limited by $\Sigma_2$ and $\tilde{\Gamma}_{p^*}$ then all arcs of trajectories $\tilde{\Gamma}_{p_i}\subset R_1$.
\end{itemize}
\end{definition}

To study the chaos and minimality on the torus and sphere we consider the following model $X=(X^+,X^-)$ of PSVF where one of them is regular (constant) and the other present a finite number of fold singularities
\begin{equation}\label{expressao-campo-regular-dobra}
\begin{array}{ll}
X^+   &=(\alpha, \beta)\\\\
X^-   & =(1, X_2^-),
\end{array}
\end{equation}
where $X_2^-$ is a smooth function such that the equation $X_2^-(x,1/2)=0$ has exactly $n$ simple solutions (these solution are the fold singularities of $X^-$ in $\Sigma_2$).


The next lemma provide us the dynamics of the sliding vector fields.

\begin{lemma}\label{lema-direcao-campo-deslizante}
Considering the model \eqref{expressao-campo-regular-dobra} the sliding vector fields defined on all segments of stable sliding region in $\Sigma_2$ is constant or identically zero.
\end{lemma}
\begin{proof}
By \eqref{expressao-campo-deslizante}, the expression of the normalized sliding vector fields in this case is 
\[
X^s(p)=(X_1^+-X_1^-,0)(p)=(\alpha-1,0).\]
Therefore if $\alpha>1, \alpha=1, \alpha<1$, respectively, then $X^s$ points to the right, is identically zero, points to the left, respectively.
\end{proof}

\begin{remark}
As our intention is study the most generic behavior in $\Omega^r$, we suppose that $\alpha \neq 1$. 
\end{remark}

\subsection{Analysis of global dynamics on $\T$}\label{Subsecao-Caos-Toro}

Now we present the main result about the global dynamics of $X\in \Omega^r$ defined in \eqref{expressao-campo-regular-dobra}.

\begin{theorem}\label{teorema-dinamica-global-toro}Consider $\M=\T$ and $X\in \Omega^r$ given by \eqref{expressao-campo-regular-dobra}. We have:

\begin{itemize}
\item[$(i)$] if $d(q)> 0(d(q)<0)$ for all $q\in \Lambda$ then all positive (negative) trajectories of $\T$ passes through $p^*$;

\item[$(ii)$] If $d(p)=0$ and $d'(p)=\pi'(p)-1\neq 0$ then $\Gamma_p$ is a hyperbolic limit cycle of $X$. Besides than, if $d'(p)<0$ $(d'(p)>0)$ then $\Gamma_p$ is an attractor (repeller) limit cycle and all trajectories of $X$, except $\Gamma_p$, passes thought $p^*$ for negative (positive) times. 
\item[$(iii)$] If $d(q_i)=d(q_j)=0,i\neq j,$ and do not exist others points between $q_i,q_j$ in $\Lambda$ such that $d(q)=0$ then $S_{ij}$, the region limited by the periodic orbits $\Gamma_{q_i}$ and $\Gamma_{q_j}$, is a minimal set.
\item[$(iv)$] If $d(q)=0$ for $q\in I_C=[q_{1}, q_{2}]\subset \Lambda$ then all trajectories of $X$ in the region between the periodic orbits $\Gamma_{q_1}$ and $\Gamma_{q_2}$ are periodic. In this case, we have a behavior like a center in the torus.
\end{itemize}
\end{theorem}

\begin{proof}
Proof of Item $(i)$: Assume that $d(p)=\pi(q)-q>0$ for all $p\in \Lambda$. In this case, we most prove that all positive trajectories pass through $p^*\in \Sigma_2$. To prove this result, we consider the dynamics of $X^-, X^s$ and the distance map $d$.

Consider $q\in (q^*, (0,1/2))\subset \Lambda$. There are two possibilities for the trajectory $\Gamma_q$: or $\Gamma_q$ intercepts the adjacent sliding region to $p^*$ and then follows to $p^*$ by the sliding vector field or $\Gamma_q$ intercepts other stable sliding region, namely $(p_i,p_{i+1})$. In the second case, the trajectory will be ejected to $\Sigma^-$ and as all arcs of trajectories of $p_i$ are in $R_1$ this trajectory will return again to the stable sliding region until reaches the adjacent stable sliding region and finally converges to $p^*$.

Besides than, for all $p\in \Sigma_2$ then $\Gamma_p$ passes trough $p^*$.

In fact, if $p\in ((0,1/2), p^*)$ then $\Gamma_p$ intercepts $\Lambda$ in a point $q\in (q^*, (0,1/2))$ and so, by previous statements, the trajectory by $p$ passes through $p^*$. If $p\in (p^*, (1,1/2))$ then $\Gamma_p$ intercepts $\Lambda$ in a point $q$. As by hypothesis $d(q)>0$ then exists a positive integer $n_0$ such that $d^{n_0}(q)\in (q^*, (0,1/2))\subset \Lambda$ and then the result follows as the previous case. 

For $q\in ((0,0), q^*)$ we have two possibilities for the dynamics of $X$: the trajectories intercepts $\Sigma_1$ in a crossing region and then follows the flow of $X^+$ reaching $\Sigma_2$ or $\Gamma_q$ intercepts $\Lambda$. In the last case, as $d(q)>0$, there exists an integer positive $n_0$ such that $d^{n_0}(q)\in (q^*, (0,1/2))\subset \Lambda$. So, all positive trajectories of $X$ pass through $p^*$.
 
The proof of Items $(ii)$ and $(iv)$ follows by the previous statements and considering the dynamics of distance map $d:\Lambda \rightarrow \Lambda$, i. e., if $d(q)=0$ this represent a periodic orbit for $X^-$ and the hypothesis $d'(q)\neq 0$ is equivalent to the periodic orbit $\Gamma_q$ to be hyperbolic.

\vs

To proof of Item $(iii)$ we have to show that the region between $\Gamma_{q_i}$ and $\Gamma_{q_j}$ is invariant by the flows of $X$, compact, non empty and does not contain a subset proper with these properties.

The invariance of $S_{ij}$ follows by the existence and uniqueness of solutions of smooth vector fields (remind that we are considering the dynamics of $X$ only in $\Sigma^-$, that is governed by $X^-\in \mathfrak{X}^r$). As $q_i\neq q_j$ then $S_{ij}$ is non empty and is compact because is a closed set in $\T$. By hypothesis, do not exist others points between $q_i,q_j$ in $\Lambda$ such that $d(q)=0$, in other words, does not exist periodic orbits in $S_{ij}$. Note that by \eqref{expressao-campo-regular-dobra},  the smooth vector field $X^-$ does not have singular points. By the classical Poincar\'e-Bendixson theorem we get that does not exist other subset proper of $S_{ij}$ that is invariant, non empty and compact. Therefore, we conclude that $S_{ij}$ is a minimal set of $X$.
\end{proof}

The next result provide us all non-trivial invariant sets for the case where $X$, given in \eqref{expressao-campo-regular-dobra}, presents a finite number of fold singularities.

\begin{corollary}
Consider the same hypothesis of Theorem \ref{teorema-dinamica-global-toro}. Let be $q_1, \dots, q_n$ the fixed points of half-first return $\pi$. If all periodic orbits $\Gamma_{q_i}, i=1, \dots, n$ of $X$ are hyperbolic then they appear alternating between attractor and repeller. Besides than, each region limited by $\Gamma_{q_i}$ and $\Gamma_{q_{i+1}}$ for $i=2, \dots, n-2$ are minimal sets of $X$, with $\omega(p)=\Gamma_{q_i}$ and $\alpha(p)=\Gamma_{q_{i+1}}$ or $\omega(p)=\Gamma_{q_{i+1}}$ and $\alpha(p)=\Gamma_{q_{i}}$.
\end{corollary}

\begin{proof}
The proof is analogously to the proof of Theorem \ref{teorema-dinamica-global-toro}, Items $(ii)$ and $(iii)$.
\end{proof}

In the following we prove that $X$, considering the hypothesis that $d(q)\neq 0$ for all $q\in \Lambda$ then $X$ is chaotic. The approach of this proof is given in \cite{BCE-ETDS}.

\begin{theorem}\label{caos-toro22}
Consider $X$ given in \eqref{expressao-campo-regular-dobra}, with $\alpha\neq 1$ and $\beta<0$. If $d(q)\neq 0$ then $X$ is chaotic on $\T$.
\end{theorem}
\begin{proof}
First of all, we prove that any $z,w\in \T$ then there exist a positive global trajectory passing through $z$ and $t_0>0$ such that $\phi_X(t_0,z)=w$.

In fact, suppose that $d(q)>0$ for all $q\in \Lambda$. The proof for the case $d(q)<0$ is analogous. By Theorem \ref{teorema-dinamica-global-toro} we get all $z\in \T$ then the positive trajectory $\Gamma_z^+$ passes through $p^*$. So, given $z,w\in \T$ then there exist $t_1,t_2>0$ such that $\phi_X(t_1,z)=p^*$ and $\phi_X(t_2,w)=p^*$. Suppose, without loss of generality,  that $t_1-t_2\geq 0$. Therefore, 
\[
\phi_X(t_1-t_2,z)=\phi_X(-t_2,\phi_X(t_1,z)=\phi_X(-t_2,p^*)=w.
\]

We prove now that $X$ is topologically transitive. 

\noindent In fact, given any non-empty sets $U,V\subset \T$ there exist $z\in U$ and $w\in V$. By the previous statements, we proved that there exist a positive trajectory connecting $z$ and $w$. So, we get the topological transitivity. 

So, remain to prove that sensitive dependence on $\T$. 

\noindent Consider $D=diam(\T), r=D/2>0$ and $a,b\in \T$ such that $d(a,b)>r$. Let be $\varepsilon>0$ and $x\in \T$. We get $y\in N_{\varepsilon}(x)\subset \T$, where $N_{\varepsilon}(x)$ is a $\varepsilon$-neighborhood of $x$ in the topology of $\T$. As we proved previously, we get that exist $t_1,t_2>0$ such that $\phi_X(t_1,x)=a$ and $\phi_X(t_2,y)=b$. Therefore, $d_H(\phi_X(t_1,x),\phi_X(t_2,y)=d(a,b)>r$ and we conclude that $X$ is chaotic in $\T$.
\end{proof}

\subsection{Analysis of global dynamics on $\Ss$}\label{Subsecao-Caos-Esfera}

 For the next lemmas, consider $\partial Q=I_1\cup I_2\cup I_3\cup I_4$, where $I_1= [0,1]\times \{0\}$, $I_2=\{1\}\times [0,1]$, $I_3=[0,1]\times \{1\}$ and $I_4=\{0\}\times [0,1]$. \\

 \begin{figure}[h]
 	\begin{center}
 		
 		\begin{overpic}[tics=10,width=8cm,bb= 0 0 370 135]{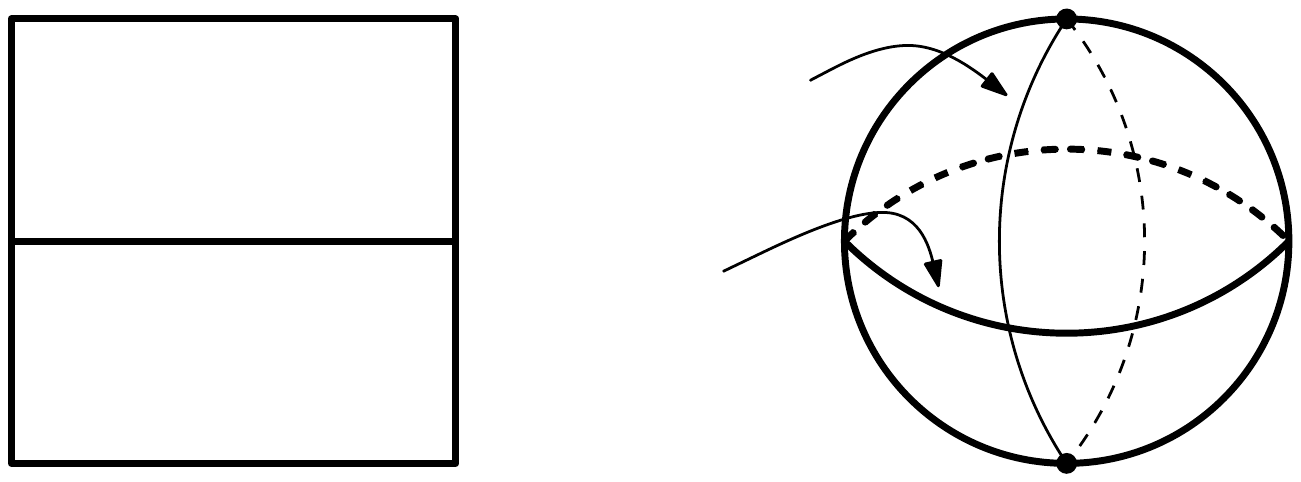}
 			\put(15,-5){$I_1$}\put(38,15){$I_2$}
 			\put(15,40){$I_3$}\put(-5,15){$I_4$}
 			\put(15,20){$\Sigma$}
 			
 			\put(52,12){$\Sigma$}
 			\put(50,27){$I_2=I_4$}
 			\put(80,-5){$I_1$}
 			\put(80,38){$I_3$}
 		\end{overpic}
 	\end{center}
 	\caption{The construction of the sphere, with the discontinuity manifold $\Sigma$.}\label{fig:figs2}
 \end{figure}
 
 \begin{lemma}\label{lema-orbitas-homoclinicas}There is a family of trajectories that are homoclinic to $p_S$.
 \end{lemma}
 \begin{proof}Consider the trajectories of $X$ viewed on square $Q$. Then the existence of the visible and invisible tangency points $p$ and $q$, respectively, implies that there are visible and invisible tangencies of $X$ on the segment $I_1$, denote by $\tilde q$ and $\tilde p$, respectively. The trajectories of $X$ in a small neighborhood of $\tilde p$ intersects $I_1$ on two points. On $\Ss=Q/\sim$ all these trajectories are closed and have $p_S$ as a common points.
 \end{proof}
 
 
%
 
As the analysis developed in previous subsection, the distinct dynamics of $X$ on $\Ss$ depends on the sign of displacement map.  In the following, we consider these different possibilities.


 Recall that $p_i\in \Sigma_2 \subset Q/\sim$ and $\eta_i\in I_1 \subset Q$ are the fold singularities of $X^-$,  $q_{i}=\Gamma_{p_{i}}\cap \Lambda$ the points where $\Gamma_{p_i}$ intercepts $\Lambda$ for the first time and $\tilde{q_i}=\pi(q_i)$ for $i=1, \dots, 2n$. We shall use the same definition of $p^*$ as defined in Definition \ref{pstar} for $\T$.


\begin{theorem}\label{teorema-dinamica-global-esfera}Consider $\M=\Ss$ and $X\in \Omega^r$ given by \eqref{expressao-campo-regular-dobra}. We have:

\begin{itemize}
\item[$(i)$] $\Ss$ can be decomposed as $\Ss=\M_s\cup\M_c\cup\M_h$, where $\M_h$ has positive measure and is foliated by trajectories homoclinic to $p_S$, $\M_c$ contains the upper hemisphere and has the property that the trajectory of every $x\in \M_c$ pass through $p^*$ and $\M_s$ is invariant and contained in the lower hemisphere.

\item[$(ii)$] if $d(q)> 0(d(q)<0)$ for all $q\in \Lambda$ then all positive (negative) trajectories of $\Ss$ of points not in $\M_h$ passes through $p^*$;

\item[$(iii)$] If $d(p)=0$ and $d'(p)=\pi'(p)-1\neq 0$ then $\Gamma_p$ is a hyperbolic limit cycle of $X$. Besides than, if $d'(p)<0$ $(d'(p)>0)$ then $\Gamma_p$ is an attractor (repeller) limit cycle and all trajectories of $X$, except for points in $\M_s\cup \M_h\cup\{p_S,p_N\}$, passes through $p^*$ for negative (positive) times. 

\item[$(iv)$] If $d(q_i)=d(q_j)=0,i\neq j,$ and do not exist others points between $q_i,q_j$ in $\Lambda$ such that $d(q)=0$ then $S_{ij}$, the region limited by the periodic orbits $\Gamma_{q_i}$ and $\Gamma_{q_j}$, is a minimal set. In this case, $\M_s$ has positive measure.

\item[$(v)$] If $d(q)=0$ for $q\in I_C=[q_{1}, q_{2}]\subset \Lambda$ then all trajectories of $X$ in the region between the periodic orbits $\Gamma_{q_1}$ and $\Gamma_{q_2}$ is periodic. In this case, $\M_s$ has positive measure and behaviors like a center.

\item [$(vi)$]In particular if $q_i,q_j\in I_1$ then there is a continuum family of trajectories connecting $p_s$ and $p^*$ (and this is a stable property).
\end{itemize}
\end{theorem}

 
 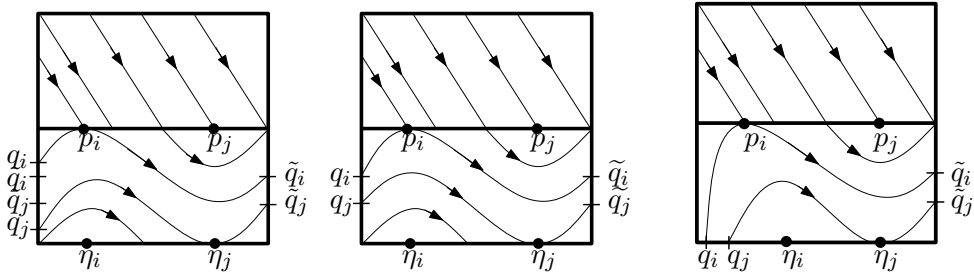
\begin{figure}[ht]
\begin{overpic}[tics=10,width=3.3cm,bb=0 0 138 132]{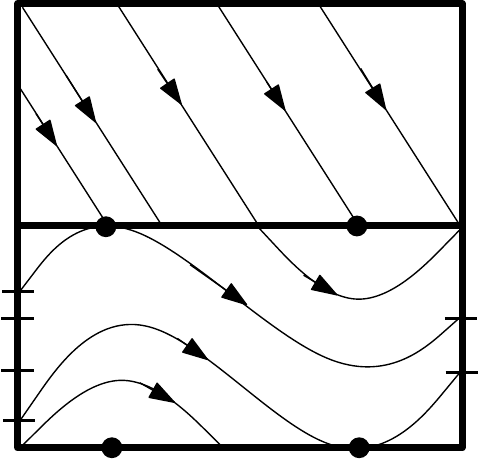}
\put(20,42){$p_i$}\put(72,42){$p_j$} \put(20,-5){$\eta_i$}\put(72,-5){$\eta_j$}
\put(-9,8){$q_j$} \put(-9,17){$\tilde q_j$}\put(103,17){$\tilde q_j$}
\put(-9,26){$\tilde q_i$}\put(103,28){$\tilde q_i$} \put(-9,35){$q_i$}
\end{overpic}\hspace{1cm}\begin{overpic}[tics=10,width=3.3cm,bb=0 0 138 132]{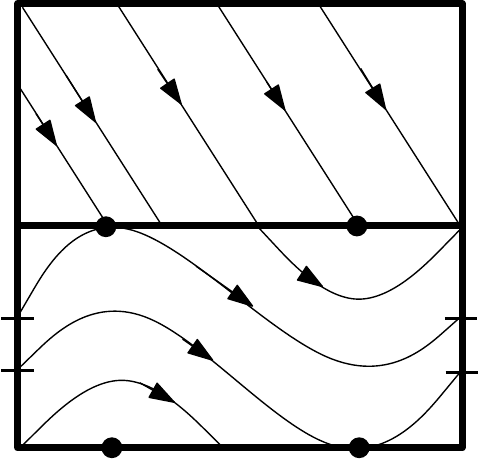}
\put(20,42){$p_i$}\put(72,42){$p_j$}\put(20,-5){$\eta_i$}\put(72,-5){$\eta_j$}
\put(-9,17){$q_j$}\put(103,17){$\widetilde{q_j}$} \put(-9,28){$q_i$} \put(103,28){$\widetilde{q}_i$}
\end{overpic}
\hspace{1cm}
		\begin{overpic}[tics=10,width=3.3cm,bb=0 0 133 133]{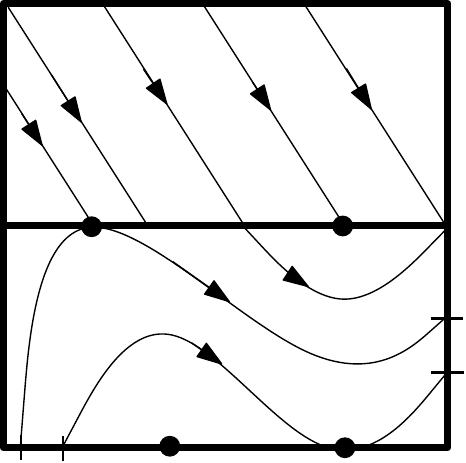}
			\put(20,42){$p_i$}\put(72,42){$p_j$}
			\put(35,-5){$\eta_i$}\put(72,-5){$\eta_j$}
			\put(13,-5){$q_j$}
			\put(103,17){$\tilde q_j$}
			\put(103,28){$\tilde q_i$}
			\put(1,-5){$q_i$}
		\end{overpic}
\caption{Cases (ii), (v) and (vi) of Theorem \ref{teorema-dinamica-global-esfera}, respectively.}\label{fig:figa5}
\end{figure}

\begin{proof}[Proof of Theorem \ref{teorema-dinamica-global-esfera}]
The proof of this Theorem is very similar to the proof of Theorem \ref{caos-toro22}. We just have to prove $(i)$. Figure \ref{fig:figa5} illustrates cases $(i), (v)$ and $(vi)$.

Let us prove (ii). Let $\eta_i$ be an invisible tangency in $I_1$ 
(as in all subfigures of Figure \ref{fig:figa5}). For $\delta>0$, consider $J_\delta=B(\eta_i,\delta)\cap I_1$ an open interval contained in $I_1$ and containing $\eta_i$. As the tangency points are isolated, we can define a first return map $P_i:\tilde J_\delta\subset J_\delta\cap \{\eta\in J_\delta,\, \eta<\eta_i\}\rightarrow J_\delta$ for some subset $\tilde J_\delta\subset J_\delta$ with the property that the flow of $X^-$ apply $\tilde J_\delta$ over $J_\delta\cap \{\eta\in J_\delta,\, \eta>\eta_i\}$. Then $\M_h^\delta=\{\phi_{X^-}(t,\eta);\, \eta\in \tilde J_\delta, t\in [0,t(\eta)]\}$, where $t(\eta)$ is the first return time of $\eta$, is foliated by trajectories homoclinic to $p_S$. Take $\delta_i$ the largest positive number with this property and denote $\M_h^{\eta_i}=\M_h^{\delta_i}$. Note that the boundary of $\M_h^{\eta_i}$  is an arc of the trajectory of $\eta_{i+1}$. Now take $\M_h=\bigcup_{\eta_l\in IT}\M_h^{\eta_l}$, where $IT$ is the set of invisible tangencies on $I_1$. Note that $\M\setminus int(\M_h)$ can be decomposed as $\M_c\cup\M_s$, in the same sense of the torus.

Note that $(vi)$ follows from the fact that if $q_i\in I_1$ then for every $\eta\in I_1$ and $\eta<q_i$, then the trajectory of $\eta$ ends at $p^*$, and as the tangencies are isolated, there are a continuum of these points.

\end{proof}

Theorem \ref{teorema-dinamica-global-esfera} (i) and (iii) implies the following result:

\begin{corollary}\label{cor-dinamica-global-esfera}Consider $\M=\Ss$ and $X\in \Omega^r$ given by \eqref{expressao-campo-regular-dobra}. Then there is a set $\M_c\subset\Ss$ of positive measure and containing $\Sigma$ such that if $x\in \M_c$, the trajectory of $x$ pass through $p^*$.
\end{corollary}

Corollary \ref{cor-dinamica-global-esfera} is an analogous of Theorem \ref{caos-toro22}, however, due to the presence of $p_N,p_S$ as singularities, with $p_N\in\M_c$, the set $\M_c$ is not minimal in the sense of Definition \ref{def-minimal}.

\section{Final remarks}\label{secao-remarks}

The qualitative theory of PSVF had a fast development after a mathema\-ti\-cal formalization, provided by the results in \cite{F}, \cite{Ko} and \cite{V}, for example. However, comparing the maturity of the theory of PSVF with the analogous to the smooth case, it is clear that has much work and effort to be done to reach the level of results and fullness of the smooth case. Besides than, in the last years the number of papers and books treating specifically of this subject has increased, see \cite{diBernardo-livro}, \cite{F}, \cite{Ko}, \cite{Or}, among others.

In this paper we started the study of PSVF in bidimensional surfaces besides then euclidean spaces. Specifically, we proved the ergodicity of the irrational rotation on torus in case where the switching manifold $\Sigma$ on torus coincides with the crossing region. In the case of sphere, the dynamics coincides with the north-south flow, i.e., for all points in sphere, except the north and south poles which are singularities, has $\omega$-limit, $\alpha$-limit as the south, north pole, respectively. When $\Sigma \equiv \Sigma^{s,u}$ and $X^s$ is regular there exists a global attractor/repeller on the torus and sphere. This kind of result can be helpful in the global study of dynamics on torus and sphere. 

When $X\in \Omega^r$ presents fold singularities, we see a drastic changes of the dynamics depending, among other factors, on the number of fold singularities (see Theorem \ref{prop-Dobra-Regular}). We have to avoid the PSVF that presents an odd number of fold singularities in $\Sigma$, as an odd number of folds implies in a line of singularities. Finally, as in the previous case, when we consider a finite number of pseudo equilibrium for $X^s$, if this number is odd, almost one of this pseudo equilibrium is a saddle-node point. Therefore, this PSVF it is not structural stable.

In Subsection \ref{Secao-Caos}, we explore the presence of chaos and minimal sets of  the PSVF on $\M$. We provide conditions under which the system presents a chaotic behavior. We also characterize when there exists minimal sets of $X\in \Omega^r$ (considering a finite number of fold singularities). 

We remark the difference between the results obtained here for PSVF and some classical results for smooth dynamics. For instance, in the case $\M=\Ss$, we prove the existence of a region on $\Ss$ composed by homoclinic connections and this behavior is generic in the set of $X\in \Omega^r$ on $\Ss$. 

The case $\M=\T$ also present some unusual results, for instance the existence of an open subset of the set of PSVF that has a transient behavior and sensibility under the initial conditions, which it is not generic.

Theorems A, B and C pave the way to obtain some Kupka-Smale results for PSVF defined on $\M$. For instance, we have to study PSVF with singularities on $\M$, that were not considered on this paper. The analysis of the dynamics of PSVF defined on sphere and torus, and also other compact manifolds, will be treated in a forthcoming paper.  



\vs

\noindent {\textbf{Acknowledgments.}} R. M. Martins is supported by FAPESP-Brazil project 2015/06903-8. D. J. Tonon is supported by grant\#2012\-/10 26 7000 803, Goi\'as Research Foundation (FAPEG), PROCAD/CAPES grant 88881.0 68462/2014-01 and CNPq-Brazil grants 478230 /2013-3 and 443302/\-2014-6. This work was partially realized at UFG/Brazil as a part of project numbers 35796 and 040393.


\begin{thebibliography}{99}



\bibitem{diBernardo-livro} {\sc M. di Bernardo, C. J. Budd, A. R. Champneys, P. Kowalczyk}, {\it Piecewise-smooth Dynamical Systems -- Theory and Applications}, Springer-Verlag, 2008.

\bibitem{diBernardo-electrical-systems} {\sc M. di Bernardo, A. Colombo, E. Fossas}, {\it Two-fold singularity in nonsmooth electrical systems}, Proc. IEEE International Symposium on Circuits ans Systems, 2713--2716, 2011.


\bibitem{B-P-S} {\sc M. E. Broucke, C. C. Pugh, S. N. Simi\'c}, {\it Structural stability of piecewise smooth systems}, Computational and Applied Mathematics, vol. \textbf{20}, 51-89, 2001.

\bibitem{Buzzi-Carvalho-Eusebio} {\sc C. Buzzi, T. de Carvalho, R. D. Euz\'{e}bio}, {\it On Poincar\'{e}-Bendixson theorem and non-trivial minimal sets in planar non smooth vector fields}, arXiv:1307.6825.

\bibitem{BCE-ETDS} {\sc C. A. Buzzi, T. de Carvalho, R. D. Euzebio},
{\it Chaotic Planar Piecewise Smooth Vector Fields With Non Trivial Minimal Sets}, Ergodic Theory of Dynamical Systems, published online: http://journals.cambridge.org/abstract\underline{ }S0143385714000674 (doi: 10.1017/etds.2014.67) (2014).

\bibitem{Tiago-Durval} {\sc T. de Carvalho, D. J. Tonon}, {\it Structural Stability and Normal forms of piecewise smooth vector fields on $\R^3$}, Publicationes Mathematicae Debrecen, v. 86, p. 1-20, 2015.

\bibitem{5} {\sc D. Chillingworth}, {\it Discontinuity geometry for an impact oscillator}. Dynam. Syst., {\bf 17} (2011), 389--420,  2011.

\bibitem{C-B-F-J} {\sc A. Colombo, M. di Bernardo, E. Fossas, M. R. Jeffrey}, {\it Teixeira singularities in 3D switched feedback control systems}, Systems and Control Letters, vol. \textbf{59}, Issue 10, pp. 615-622, 2010.


\bibitem{F} {\sc A. F. Filippov}, {\it Differential equations with discontinuous righthand sides}, vol. {\bf 18} of Mathematics and its Applications (Soviet Series), Kluwer Academic Publishers Group, Dordrecht, 1988.

\bibitem{G-S-T} {\sc M. Guardia, T. M. Seara, M. A. Teixeira}, {\it Generic bifurcations of low codimension of planar Filippov Systems}, Journal of Differential Equations, vol. \textbf{250}, 1967-2023, 2011.

\bibitem{katok} {\sc B. Hasselblatt, A. Katok}, {\it Introduction to the Modern Theory of Dynamical Systems}, Cambridge University Press, 1995.

\bibitem{J-T-T1} {\sc A. Jacquemard, M. A. Teixeira, D. J. Tonon}, {\it Stability conditions in piecewise smooth dynamical systems at a two-fold singularity}, Journal of Dynamical and Control Systems, vol. \textbf{19}, 47-67, 2013.

\bibitem{J-T-T2} {\sc A. Jacquemard, M. A. Teixeira, D. J. Tonon}, {\it Piecewise smooth reversible dynamical systems at a two-fold singularity}, International Journal of Bifurcation and Chaos, vol. \textbf{22}, 2012.

\bibitem{6} {\sc A. Jacquemard, D. J. Tonon}, {\it Coupled systems of non-smooth differential equations}. Bulletin des Sciences Math\'ematiques {\bf 136}, 239--255, 2012.

\bibitem{J-C} {\sc M. R. Jeffrey, A. Colombo}, {\it The two-fold singularity of discontinuous vector fields}, SIAM J. Appl. Dyn. Syst., vol. \textbf{8}, 624-640, 2009.

\bibitem{J-C2} {M. R. Jeffrey, A. Colombo}, {\it Nondeterministic Chaos, and the Two-fold Singularity in Piecewise Smooth Flows }, SIAM J. Appl. Dyn. Syst. vol. \textbf{10}, Issue 2, pp. 423-451, 2011.


\bibitem{Ko} {V. S. Koslova}, {\it Roughness of a discontinuous 	system}, Vestnik Moskovskogo Universiteta Seriya 1 Matematika Mekhanika, vol.{\bf 5}, 16-20, 1984.


\bibitem{Kuz} {\sc Yu. A. Kuznetsov, S. Rinaldi, A. Gragnani}, {\it One-parameter bifurcations in planar Filippov Systems}, Int. Journal of Bifurcation and Chaos, vol. \textbf{13}, 2157--218 8, 2003.

\bibitem{Makarenkov-Lamb} {\sc O. Makarenkov, J. S. W. Lamb}, {\it Dynamics and bifurcations of non smooth systems: A survey}, Physica D: Nonlinear Phenomena, vol. \textbf{241}, 1826-1844, 2012.

\bibitem{Or} {\sc Y. V. Orlov}, {\it Discontinuous Systems-Lyapunov Analysis and Robust Synthesis under Uncertainty Conditions}, Springer, 2009.



\bibitem{T1} {\sc M. A. Teixeira}, {\it Stability conditions for discontinuous vector fields}, Journal of Differential Equations, vol. \textbf{88}, 15-29, 1990.

\bibitem{T3} {\sc M. A. Teixeira}, {\it Generic bifurcations of sliding  vector fields}, Journal of Mathematical Analysis and Applications, vol. {\bf 176}, 436-457, 1993.

\bibitem{Marco-enciclopedia} {\sc M. A. Teixeira}, {\it Perturbation Theory for Non-smooth Systems}, Meyers: Encyclopedia of Complexity and Systems Science, vol.\textbf{152}, 2008.


\bibitem{V} {S. M. Vishik}, {\it Vector Fields near the boundary of a manifold}, Vestnik, Moskov, Univ. Serv. I, Mat. Meh., {\bf27}, 1, 21-28, 1972.


\end{thebibliography}
\end{document}